\RequirePackage[l2tabu,orthodox,abort]{nag}

\documentclass[
12pt,
a4paper,
reqno,
tbtags
]
{amsart}

\usepackage[all,error]{onlyamsmath}

\usepackage[a4paper,left=3cm,right=3cm,top=3cm,bottom=3cm]{geometry}

\usepackage[T1]{fontenc}
\usepackage[utf8x]{inputenc}

\usepackage[strict=true]{csquotes}

\usepackage{graphicx}
\usepackage{amsmath}
\usepackage{amsfonts}
\usepackage{amssymb}
\usepackage{amsthm}
\usepackage{bbm}
\usepackage{color}
\usepackage{marginnote}

\usepackage{subfig}

\usepackage{enumitem}

\makeatletter
\def\namedlabel#1#2{\begingroup
    #2%
    \def\@currentlabel{#2}%
    \phantomsection\label{#1}\endgroup
}
\makeatother

\numberwithin{equation}{section}

\theoremstyle{plain}

\newtheorem{theorem}{Theorem}[section]
\newtheorem{proposition}[theorem]{Proposition}
\newtheorem{lemma}[theorem]{Lemma}
\newtheorem{corollary}[theorem]{Corollary}

\theoremstyle{definition}

\newcommand{\R}{\mathbb{R}}

\newcommand{\C}{\mathbb{C}}

\newcommand{\X}{\mathbb{X}}

\newcommand{\lin}{\mathcal{L}}

\newcommand{\cA}{\mathcal{A}}

\newcommand{\cg}{\mathcal{G}}

\newcommand{\cI}{\mathcal{I}}

\newcommand{\cP}{\mathcal{P}}
\newcommand{\cQ}{\mathcal{Q}}
\newcommand{\cG}{\mathcal{G}}

\newcommand{\GG}{\mathsf{G}}
\newcommand{\VV}{\mathsf{V}}

\newcommand{\EE}{\mathsf{E}}

\newcommand{\ee}{\mathsf{e}}
\newcommand{\ff}{\mathsf{f}}
\newcommand{\vv}{\mathsf{v}}
\newcommand{\sv}{\mathsf{v}}

\newcommand{\eE}{\ee \in \EE}

\newcommand{\lle}{l_{\ee}}
\newcommand{\rre}{r_{\ee}}
\newcommand{\llef}{l_{\ee,\ff}}
\newcommand{\rref}{r_{\ee,\ff}}

\newcommand{\ffel}{\ff_{\ee}^-}
\newcommand{\ffer}{\ff_{\ee}^+}

\newcommand{\eeinit}{\ee_{\mathrm{init}}}
\newcommand{\eeterm}{\ee_{\mathrm{term}}}

\newcommand{\pee}{p_{\ee}}
\newcommand{\qee}{q_{\ee}}

\DeclareMathOperator{\dom}{\mathcal{D}}
\DeclareMathOperator{\res}{\mathcal{R}}

\renewcommand{\phi}{\varphi}

\newcommand{\e}{\mathrm{e}}

\newcommand{\dd}[1]{\mathop{}\!\mathrm{d}#1}

\newcommand{\lino}[1]{\mathcal{L}(#1)}
\newcommand{\set}[1]{\lrc{#1}}

\DeclareMathOperator{\re}{Re}

\DeclareMathOperator{\range}{range}

\DeclareMathOperator{\Int}{int}

\usepackage{mathtools}
\DeclarePairedDelimiter{\lrp}{\lparen}{\rparen}
\DeclarePairedDelimiter{\lrb}{\lbrack}{\rbrack}
\DeclarePairedDelimiter{\lrc}{\lbrace}{\rbrace}
\DeclarePairedDelimiter{\abs}{\lvert}{\rvert}
\DeclarePairedDelimiter{\norm}{\lVert}{\rVert}

\makeatletter
\newcommand{\normt}{\@ifstar\@normts\@normt}
\newcommand{\@normts}[1]{%
  \left|\mkern-1.5mu\left|\mkern-1.5mu\left|
   #1
  \right|\mkern-1.5mu\right|\mkern-1.5mu\right|
}
\newcommand{\@normt}[2][]{%
  \mathopen{#1|\mkern-1.5mu#1|\mkern-1.5mu#1|}
  #2
  \mathclose{#1|\mkern-1.5mu#1|\mkern-1.5mu#1|}
}
\makeatother

\usepackage{xspace}

\renewcommand{\leq}{\leqslant}

\usepackage[many]{tcolorbox}

\newcounter{myhypo}
\renewcommand\themyhypo{(H\arabic{myhypo})}

\newtcolorbox{hypo}[1][]{
  breakable,
  enhanced,
  before={\medskip\noindent},
  after={},
  top=0pt,
  bottom=0pt,
  colback=white,
  boxrule=0pt,
  boxsep=0pt,
  left=40pt,
  right=0pt,
  leftrule=0pt,
  colframe=white,
  outer arc=0pt,
  overlay={
    \node[inner sep=0pt,anchor=west]
    at (frame.west)
    {\refstepcounter{myhypo}\themyhypo\label{#1}};
  }, 
}

\usepackage{hyperref}

\title[]{Sticky diffusions on graphs}
\author[A. Gregosiewicz]{Adam Gregosiewicz}

\address{%
  Lublin University of Technology,\ ul.~Nadbystrzycka 38A,\ 20-618 Lublin,
  Poland}

 \email{a.gregosiewicz@pollub.pl}


\begin{document}

\begin{abstract}
We consider diffusion processes on metric graphs with semipermeable
\emph{sticky} membranes in each vertex.
We prove that the process is governed by a Feller semigroup and find its
asymptotic behavior as diffusion’s speed increases to infinity with the same
rate as permeability coefficients decreases to zero.
\end{abstract}

\maketitle

\section{Introduction}
\label{sec:introduction}

Since around 1980 numerous papers have been published on the topic of evolution
operators acting on metric graphs -- see for example~\cite{kramar2019,mugnolo},
or recent survey~\cite{kramar2020} and references given there.
In this context operators related to the diffusion process are one of the most
extensively examined.
More specifically, let \( \cg \) be a~finite graph without loops, and assume
that there is a~Markov process on \( \cg \), which on each edge behaves like a
Brownian motion with given variance.
Moreover, suppose that each vertex of the graph is a~semipermeable membrane with
given permeability coefficients, that is for each vertex there are nonnegative
numbers, describing the possibility of a~particle passing through the membrane
from an edge \( \ee \) to \( \ff \).
In 2012, A.~Bobrowski and K.~Morawska~\cite{bobrowski-morawska} considered such
diffusion processes on a simple graph to model synaptic depression dynamic.
The results were generalized by Bobrowski in~\cite{bobrowski-diff} to the case
of arbitrary graphs.
He proved that these processes are governed by strongly continuous semigroups of
operators (see, for example,~\cite{MR1721989} for an introduction to the theory
of semigroup of linear operators), and, moreover, that if the diffusion's speed
increases to infinity with the same rate as permeability coefficients decreases
to zero (this is an example of a \emph{small parameter}, or \emph{singular
  perturbation}, problem), then there is a~limit process, in the sense of
Theorem~3.26 in~\cite{liggett2010}, which behaves like a~Markov chain on the
vertices of the line graph of \( \cg \), see Figure~\ref{fig:graph}.
\begin{figure}[b]
\centering
\includegraphics{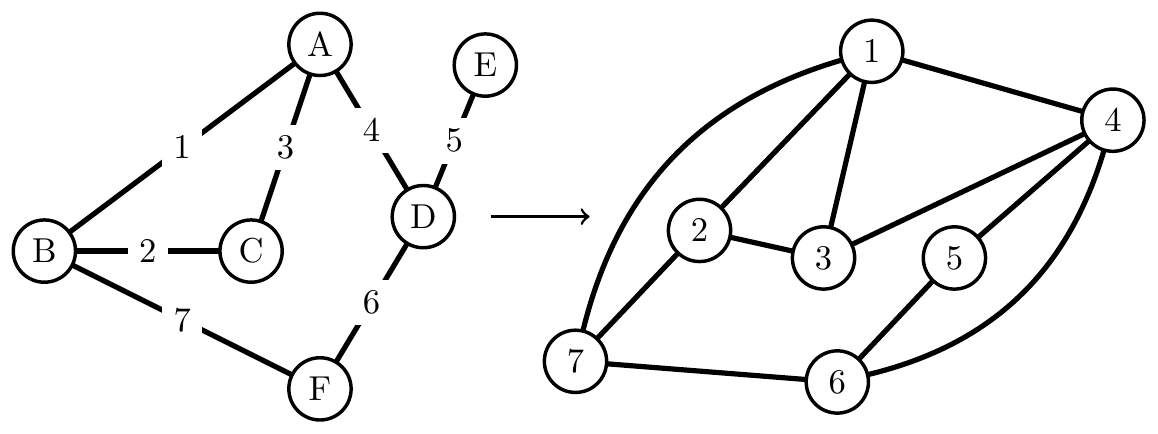}
\caption{Diffusion on a graph \( \cg \) becomes a Markov chain on the vertices
  of the line graph of \( \cg \).}
\label{fig:graph}
\end{figure}
Because in Bobrowski's papers the analysis takes place in the space of
continuous functions on \( \cg \), the related semigroups describe dynamics of
(weighted) conditional expected values of these processes.
It is also worth noting that the communication between edges is based on the
Fick law, or, in other words, that boundary conditions at vertices are
Robin-type.

In 2014, in order to obtain the dynamics of densities of Bobrowski's processes
distributions, we considered in~\cite{gregosiewicz2014} a ``dual'' description of
the processes with the underlying space being the \( L^1 \)-type space of
Lebesgue integrable functions on \( \cg \).
One may wish to mimic the argument of the continuous case but this is not fully
possible (the reason is that a pointwise evaluation is not a bounded functional
in an \( L^1 \)-type space), thus a different method is needed.

These results, in the space of continuous functions and in the \( L^1 \)-type
space, were improved by J.~Banasiak \emph{et
  al.}~\cite{banasiak2016b,banasiak2016}.
Moreover, if we are interested in generation theorems only, then Banasiak's
results follow from a recent and very general scheme developed by T.~Binz and
K.J.~Engel (in the space of continuous functions), and M.~Kramar Fijav\v{z} and
K.J.~Engel (in the \( L^1 \)-type space).


In this paper we generalize results of~\cite{bobrowski-diff} to the case of
\emph{sticky boundary conditions}.
We consider the process that on each edge of the underlying graph behaves as a
Brownian motion, and in each vertex we put a semipermeable \emph{sticky}
membrane; see Example~3.59 in~\cite{liggett2010} or Appendix for detailed
description of this type of boundary condition.
We prove, see Theorem~\ref{thm:main-feller}, that the process is governed by a
Feller semigroup and, see Theorem~\ref{thm:convergence-in-C}, find its
asymptotic behavior as diffusion’s speed increases to infinity with the same
rate as permeability coefficients decreases to zero.

\section{Sticky diffusion on a graph}
\label{sec:sticky-diffusion-on-graph}

Let \( \GG = (\VV,\EE) \) be a finite graph without loops, where
\( \VV = \set{\sv_1,\ldots,\sv_n} \) is the set of vertices and and
\( \EE = \set{\ee_1,\ldots,\ee_m} \) is the set of edges of \( \GG \).
We arbitrarily fix an orientation of \( \GG \) and introduce the \emph{incidence
  matrix} \( \cI = (\iota_{\vv,\ee})_{\vv \in \VV,\ee \in \EE} \) defined by
\[
\iota_{\vv,\ee} =
\begin{cases}
-1, & \text{\( \vv \) is the initial endpoint \( \eeinit \) of \( \ee \)}, \\
+1, & \text{\( \vv \) is the terminal endpoint \( \eeterm \) of \( \ee \)}, \\
0, & \text{otherwise}.
\end{cases}
\]
To define a metric analogue of the discrete object \( \GG \), we assign to each
edge \( \ee \in \EE \) a closed interval, which we normalize to \( [0,1] \) for
simplicity.
We let \( \cG \) to be the disjoint union of these intervals, that is
\[
\cG \coloneqq \bigsqcup_{\ee \in \EE} [0,1],
\]
and denote elements of \( \cG \) by \( \ee(s) \) for all \( \ee \in \EE \) and
\( s \in [0,1] \).
Notice that there can be many `copies' of a single vertex \( \sv \in \VV \) in
\( \cG \), and that \( \cG \) does not itself posses a graph structure -- there
is no information about connections between vertices.
From this point of view, if we want to treat \( \cG \) as a metric graph, we
need to refer to adjacency matrix \( \cI \), for example.
If we endow each interval of \( \cG \) with the standard metric on \( [0,1] \),
then \( \cG \) becomes a disconnected compact metric space.
With slight abuse of notation, we view \( \ee \) as an element of \( \EE \) as
well as a connected component of \( \cG \).
We parametrize each component according to the orientation of the related edge
\( \ee \), and call \( \ee^- \coloneqq \ee(0) \) and
\( \ee^+ \coloneqq \ee(1) \) the \emph{left} and \emph{right endpoints} of (the
interval related to) \( \ee \).
By \( C(\cG) \) we denote the space of continuous function on \( \cG \) with the
standard supremum norm.
This space is isometrically isomorphic to the Cartesian product
\( \bigtimes_{\eE} C[0,1] \), where \( C[0,1] \) is the space of continuous
functions on \( [0,1] \).
Therefore, we may identify \( u \in C(\cG) \) with \( (u_{\ee})_{\eE} \), where
\( u_{\ee} \) is a member of \( C[0,1] \).
Nevertheless, whenever possible, we consider \( u \in C(\cG) \) as a real-valued
function on a disconnected space \( \cG \), and use the edgewise identification
\( u = (u_{\ee})_{\eE} \) when needed.
Such approach is convenient, since when we apply the positive maximum principle
for Feller semigroups, we have to work in the space of real-valued continuous
functions.
Note also that it makes sense to speak about differentiable functions on
\( \cG \), and in particular, by \( C^k(\cG) \) we denote the space of
\( k \)-times continuously differentiable functions on \( \cG \).

Fix \( \epsilon \in (0,1] \), and for every \( \eE \) let \( \sigma_{\ee} \) be
a positive constant.
We define \( \sigma \in C(\cG) \) to be the continuous function on \( \cG \),
which on each edge \( \eE \) equals \( \sigma_{\ee} \), and consider the
diffusion equation
\[
\partial_t u(t,x) = \epsilon \sigma(x) \partial_{xx} u(t,x), \qquad
t > 0,\ x \in \Int \cG,
\]
where \( \Int \cG \) is the interior of \( \cG \), which is the disjoint union of
the intervals \( (0,1) \), that is
\[
\Int \cG = \bigsqcup_{\eE} (0,1).
\]
This means that on the edge \( \ee \) the related process behaves like a
Brownian motion with variance \( \epsilon \sigma_{\ee} \).

To describe the communication between the edges, we
follow~\cite{banasiak2016b,banasiak2016,bobrowski-diff,gregosiewicz2020}, and
for each \( \eE \) we let \( l_{\ee} \) and \( r_{\ee} \) be nonnegative real
numbers giving the rates at which Brownian particles pass through the membrane
from the edge \( \ee \) to the edges incident in the left and right endpoints,
respectively.
Also, for \( \ee,\ff \in \EE \) such that \( \ee \neq \ff \) let \( \llef \) and
\( \rref \) be nonnegative real numbers satisfying
\( \sum_{\ff \neq \ee} \llef \leq \lle \) and
\( \sum_{\ff \neq \ee} \rref \leq \rre \); the summation here is taken over all
\( \ff \in \EE \) such that \( \ff \neq \ee \).
These numbers determine the probability that \emph{after} filtering through the
membrane of the edge \( \ee \) a particle will enter the edge \( \ff \).
More specifically, the probability that a particle after filtering through the
membrane at the left endpoint \( \ee^- \) will enter the edge \( \ff \) equals
\( \llef/\lle \), and, by default, if \( \ff \) is not incident with the initial
endpoint of \( \ee \) in \( \GG \), we put \( \llef = 0 \).
Analogously, \( \rref/\rre \) is the probability that after filtering through
the membrane at \( \ee^+ \) the particle will enter the edge \( \ff \), and
\( \rref = 0 \), provided that \( \ff \) is not incident with the terminal
endpoint of \( \ee \) in \( \GG \).
For \( \ee \neq \ff \) we denote by \( \ffel \) and \( \ffer \) the left and
right, respectively, endpoint of \( \ee \) seen as an endpoint of \( \ff \).
In other words,
\[
\ffel =
\begin{cases}
\ff^-, & \ff_{\mathrm{init}} = \eeinit, \\
\ff^+, & \ff_{\mathrm{term}} = \eeinit, \\
\text{undefined}, & \text{\( \ff \) is not incident to \( \eeinit \)},
\end{cases}
\]
and similarly for \( \ffer \).

With these notations we impose for each \( \eE \) the following sticky
transmission conditions (see Appendix for a detailed description) between the
edges:
\[
p_{\ee} \partial_{xx} u(t,\ee^-) - p_{\ee}^{*} \partial_x u(t,\ee^-) =
\sum_{\ff \in \EE} \epsilon \llef u(t,\ffel) - \epsilon \lle u(t,\ee^-),
\]
and
\[
q_{\ee} \partial_{xx} u(t,\ee^+) + q_{\ee}^{*} \partial_x u(t,\ee^+) =
\sum_{\ff \in \EE} \epsilon \rref u(t,\ffer) - \epsilon \rre u(t,\ee^+),
\]
where \( p_{\ee}, q_{\ee} \in [0,1] \), \( p_{\ee}^{*} \coloneqq 1-p_{\ee} \),
\( q_{\ee} \coloneqq 1-q_{\ee} \).
By convention, if \( \ffel \) or \( \ffer \) is not defined, or, equivalently,
\( \llef = 0 \) or \( \rref = 0 \), we let \( \llef u(t,\ffel) = 0 \) or
\( \rref u(t,\ffer) = 0 \), respectively.

\subsection{Generation theorem}
\label{sec:generation-theorem}

We begin by considering the abstract operator related to the diffusion process.
To this end, let \( \epsilon \in (0,1] \), \( \sigma = (\sigma_{\ee})_{\eE} \)
be a positive continuous function on \( \cG \) that is constant on each edge,
and \( \pee, \qee \in [0,1] \).
We let \( A_{\epsilon} \) to be the operator in \( C(\cG) \) given by
\begin{equation}
\label{eq:Aeps}
A_{\epsilon} u \coloneqq \epsilon^{-1} \sigma u'', \qquad u \in \dom(A_{\epsilon}).
\end{equation}
Recall that we consider \( C(\cG) \) as the space of real-valued continuous
functions on the disconnected space \( \cG \), hence,
\( A_{\epsilon} u(x) = \epsilon^{-1} \sigma(x) u''(x) \) for all
\( x \in \cG \).
We define the domain of \( A_{\epsilon} \) as
\begin{equation}
\label{eq:Adom}
\dom(A_{\epsilon}) \coloneqq \set{u \in C^2(\cG)\colon L u = \epsilon\Phi u},
\end{equation}
where \( L\colon C^2(\cG) \to \R^{2 \abs{ \EE }} \) (\( \abs{ E } \) is the
cardinality of \( E \)) is given by
\[
L u \coloneqq \lrp[\big]{p_{\ee} u''(\ee^-) - (1-p_{\ee}) u'(\ee^-),q_{\ee}
  u''(\ee^+) + (1-q_{\ee}) u'(\ee^+)}_{\eE}, \qquad u \in C^2(\cG),
\]
and \( \Phi\colon C(\cG) \to \R^{2 \abs{ \EE }} \) is the bounded linear
operator
\[
\Phi u = (\Phi_{\ee}^-,\Phi_{\ee}^+)_{\eE} \coloneqq \lrp[\bigg]{ \sum_{\ff \in \EE} \llef
  u(\ffel) - \lle u(\ee^-), \sum_{\ff \in \EE} \rref u(\ffer) - \rre u(\ee^+)
}_{\eE}
\]
for all \( u \in C(\cG) \).
We also denote
\[
\Phi_{\epsilon} \coloneqq \epsilon \Phi.
\]

\begin{theorem}
\label{thm:main-feller}
The operator \( A_\epsilon \) given by~\eqref{eq:Aeps}--\eqref{eq:Adom}
generates a Feller semigroup in \( C(\cG) \).
The semigroup is conservative if and only if
\begin{equation}
\label{eq:conservative}
\sum_{f \in \EE} \llef = \lle \qquad \text{and} \qquad \sum_{f \in \EE} \rref =
\rre
\end{equation}
for all \( \eE \).
\end{theorem}

Before we prove Theorem~\ref{thm:main-feller} we introduce the operator
\( A_0 \) as \( \epsilon A_\epsilon \) with \( \Phi = 0 \), that is
\( A_0 u = \sigma u'' \) for all \( u \in \dom(A_0) \), where
\[
\dom(A_0) \coloneqq \lrc{ u \in C^2(\cG)\colon L u = 0 }.
\]
In other words, on each edge \( \ee \), \( A_0 \) is a (rescaled by
\( \sigma_{\ee} \)) copy of the generator \( G_{\pee,\qee} \) of a
one-dimensional sticky diffusion on \( [0,1] \) -- see Appendix.

\begin{proof}
By a well-known characterisation of Feller semigroups, see Theorem~2.2
in~\cite{MR838085}, the operator \( A_{\epsilon} \) is a Feller generator if and
only if it is densely defined, satisfies the positive maximum principle, and the
range of \( \lambda - A_{\epsilon} \) is \( C(\cG) \) for some
\( \lambda > 0 \).
The fact that the domain \( \dom(A_{\epsilon}) \) is dense in \( C(\cG) \)
follows by Lemma~\ref{lem:dense-derivative}, and arguing as
in~\cite{bobrowski-convergence} p.~17, it is easy to check that \( A_\epsilon \)
satisfies the positive maximum principle.
Hence, we are left with proving the range condition.
To this end we use Greiner's idea of perturbing the boundary conditions,
see~\cite[Lemma~1.4]{greiner}.

Let \( \Lambda_\epsilon \) be the operator in \( C(\cG) \) defined as
\( A_\epsilon \) with domain \( C^2(\cG) \), that is
\( \Lambda_\epsilon u \coloneqq \epsilon^{-1} \sigma u'' \) for all
\( u \in C^2(\cG) \), and define
\[
L_{\lambda,\epsilon}\colon \R^{2 \abs{ \EE }} \to \ker(\lambda - \Lambda_\epsilon), \qquad L_{\lambda,\epsilon} \coloneqq (L_{|\ker(\lambda - \Lambda_\epsilon)})^{-1}.
\]
That is, \( L_{\lambda,\epsilon} \) is the inverse of \( L \) as restricted to
\( \ker(\lambda - \Lambda_\epsilon) \).
By Lemma~\ref{lem:det} in Appendix such inverse exists for all \( \lambda > 0 \) and
\[
\lrp[\big]{L_{\lambda,\epsilon}(a_{\ff},b_{\ff})_{\ff \in \EE}}(\ee(s)) =
u_{\ee}(\ee(s)) \coloneqq c_{\ee} \e^{\gamma_{\ee} s} + d_{\ee}
\e^{-\gamma_{\ee} s}
\]
for all \( (a_{\ff},b_{\ff})_{\ff \in \EE} \in \R^{2 \abs{ \EE }} \),
\( \ee \in \EE \), and \( s \in [0,1] \), where
\( \gamma_{\ee} \coloneqq \sqrt{\lambda \epsilon/\sigma_{\ee}} \), and
\( c_{\ee} \), \( d_{\ee} \) is the unique pair of real numbers
satisfying~\eqref{eq:system} for
\[
\mu \coloneqq \gamma_{\ee}, \qquad a(\mu) \coloneqq a_{\ee} \gamma_{\ee}^{-1},
\qquad b(\mu) \coloneqq b_{\ee} \gamma_{\ee}^{-1}.
\]
Moreover, by~\eqref{eq:det-estimate} there exist \( \lambda_0 > 0 \) and
\( M > 0 \) (that depend merely on \( \epsilon \) and \( \sigma \)) such that
\[
\norm{ u_{\ee}(\ee(\cdot)) }_{C[0,1]} \le \frac{M}{\sqrt{\lambda}} (\abs{
  a_{\ee} } + \abs{ b_{\ee} }), \qquad \eE,
\]
provided that \( \lambda > \lambda_0 \).
This proves that \( L_{\lambda,\epsilon} \) is a bounded operator
\( \R^{2 \abs{ \EE }} \to C(\cG) \) and its norm satisfies
\begin{equation}
\label{eq:Lnorm}
\norm{ L_{\lambda,\epsilon} } \le \frac{M'}{\sqrt{\lambda}}, \qquad \lambda > \lambda_0
\end{equation}
for some \( M' > 0 \) depending on the norm introduced in
\( \R^{2 \abs{ \EE }} \).
Therefore, since \( \Phi_{\epsilon} \) is bounded in \( C(\cG) \), there exists
\( \lambda_1 > 0 \) such that the norm of the bounded operator
\( L_{\lambda,\epsilon} \Phi_{\epsilon} \) in \( C(\cG) \) is less then \( 1 \)
for \( \lambda > \lambda_1 \), and consequently the operator
\( I_{C(\cG)} - L_{\lambda,\epsilon} \Phi_{\epsilon} \), where \( I_{C(\cG)} \)
is the identity operator in \( C(\cG) \), is invertible.

Let \( v \in C(\cG) \), and choose \( \lambda > \lambda_1 \).
Set \( w \coloneqq (\lambda - A_0)^{-1}v \), which exists by
Theorem~\ref{thm:feller-semigroup-interval}, and define
\[
u \coloneqq (I_{C(\cG)} - L_{\lambda,\epsilon} \Phi_{\epsilon})^{-1} w.
\]
We show that \( u \) belongs to \( \dom(A_\epsilon) \) and that
\( (\lambda - A_\epsilon)u = v \).
We have \( w \in \dom(A_0) \), and hence
\( u = w + L_{\lambda,\epsilon} \Phi_{\epsilon} u \) belongs to \( C^2(\cG) \).
To check that \( u \in D(A_\epsilon) \), note that
\[
Lu = Lw + L L_{\lambda,\epsilon} \Phi_{\epsilon} u = 0 + \Phi_{\epsilon} u =
\Phi_{\epsilon} u.
\]
Finally,
\begin{align*}
A_\epsilon u &= \Lambda_\epsilon w + \lambda L_{\lambda,\epsilon} \Phi_{\epsilon}
u - (\lambda - \Lambda_\epsilon)L_{\lambda,\epsilon} \Phi_{\epsilon} u\\
&= -(\lambda - \Lambda_\epsilon)w + \lambda w + \lambda L_{\lambda,\epsilon}
\Phi_{\epsilon} u = -v + \lambda u,
\end{align*}
which completes the proof of the generation part.

The semigroup generated by \( A_{\epsilon} \) is conservative if and only if the
function \( 1_{\cG} \) that equals \( 1 \) for every \( x \in \cG \) belongs to
the domain of \( A_{\epsilon} \).
Since \( L1_{\cG} = 0 \), \( 1_{\cG} \in \dom(A_{\epsilon}) \) is equivalent to
\( \Phi 1_{\cG} = 0 \), which is exactly~\eqref{eq:conservative}.
\end{proof}

As a~by-product of the proof we obtain the formula for the resolvent
\( \res(\lambda,A_{\epsilon}) \) of \( A_\epsilon \) in terms of the resolvent
\( \res(\lambda,A_0) \) of \( A_0 \).

\begin{corollary}
\label{cor:aeps-res}
For sufficiently large \( \lambda > 0 \) we have
\[
\res(\lambda,A_\epsilon)u = (I_{C(\cG)} - L_{\lambda,\epsilon} \Phi_{\epsilon})^{-1} \res(\lambda,A_0)u, \qquad u
\in C(\cG),
\]
where
\( L_{\lambda,\epsilon} \coloneqq (L_{|\ker(\lambda-\Lambda_\epsilon)})^{-1} \)
for
\[
\Lambda_\epsilon u \coloneqq \epsilon^{-1} \sigma u'', \qquad u \in C^2(\cG).
\]
\end{corollary}

\subsection{Convergence}
\label{sec:convergence}

Here we examine what happens with the semigroup generated by \( A_\epsilon \),
when \( \epsilon \) converges to \( 0^+ \).
To this end we recall (a special case of) the singular convergence theorem due to
T.~Kurtz (see~\cite[Corollary~7.7,~p.~40]{MR838085} or~\cite[Theorem~42.2 and
Theorem~7.1]{bobrowski-convergence}).

\begin{theorem}
\label{thm:kurtz-res}
Assume that the following conditions hold.
\begin{enumerate}
  \item\label{item:kurtz1} For every \( \epsilon \in [0,1] \) each operator
\( \cA_\epsilon \) is the generator of a strongly continuous semigroup in a
Banach space \( \X \), and there exists \( M > 0 \) such that
\( \norm{ \e^{t\cA_\epsilon} }_{\lin(\X)} \le M \) for all
\( \epsilon \in (0,1] \) and \( t \ge 0 \).
  \item\label{item:kurtz2} For some \textnormal{(}hence all\textnormal{)}
\( \lambda > 0 \) the resolvent \( R(\lambda,\epsilon \cA_\epsilon) \) converges
strongly to the resolvent \( R(\lambda,\cA_0) \) as \( \epsilon \to 0^+ \).
  \item\label{item:kurtz3} For every \( x \in \X \) the limit
\[
\cP x \coloneqq \lim_{t \to +\infty} \e^{t \cA_0} x
\]
exists.
  \item\label{item:kurtz4} An operator \( \cQ \) generates a strongly continuous semigroup in the space
\[
\X_0 \coloneqq \range \cP,
\]
and for some \textnormal{(}hence all sufficiently large\textnormal{)}
\( \lambda > 0 \) we have
\[
\lim_{\epsilon \to 0^+} R(\lambda,\cA_\epsilon) x = R(\lambda,\cQ) \cP x, \qquad
x \in \X.
\]
\end{enumerate}
Then
\[
\lim_{\epsilon \to 0^+} \e^{t \cA_\epsilon} x = \e^{t \cQ} \cP x
\]
for all \( t > 0 \) and \( x \in \X \).
The convergence is uniform in \( t \) on compact subsets of \( (0,+\infty) \),
and if \( x \in \X_0 \), then the formula holds also for \( t = 0 \), and the
convergence is uniform in \( t \) on compact subsets of \( [0,+\infty) \).
\end{theorem}

We apply Kurtz's theorem in the following setup.
Let \( \X \coloneqq C(\cG) \) and \( \cA_\epsilon \coloneqq A_\epsilon \) for
every \( \epsilon \in [0,1] \).
Then condition~\ref{item:kurtz1} holds by Theorem~\ref{thm:main-feller}.
Moreover, by Theorem~\ref{thm:feller-semigroup-interval}, we obtain
\begin{equation}
\label{eq:Azerolim}
\lim_{t \to +\infty} \e^{tA_0} u = \lim_{\lambda \to 0^+} \lambda
\res(\lambda,A_0) u = P u, \qquad u \in C(\cG),
\end{equation}
for
\begin{equation}
\label{eq:P}
P \coloneqq (P_{p_{\ee},q_{\ee}})_{\eE},
\end{equation}
establishing~\ref{item:kurtz3}.
Therefore, we are left with choosing appropriate \( \cQ \) and
proving~\ref{item:kurtz2}, \ref{item:kurtz4}.

To this end for each \( \eE \) we define the operator
\( Q_{\e}\colon C(\cG) \to C[0,1] \) in the following way.
If \( p_{\ee} = 1 \) and \( q_{\ee} = 1 \), then
\[
Q_{\ee} u(x) \coloneqq (1-x) \Phi_{\ee}^- u + x \Phi_{\ee}^+ u, \qquad u \in
C(\cG),\ x \in [0,1].
\]
On the other, if \( p_{\ee} \neq 1 \) or \( q_{\ee} \neq 1 \), then
\[
Q_{\ee} u \coloneqq \frac 1{1-p_{\ee}q_{\ee}} (q_{\ee}^* \Phi_{\ee}^- u +
p_{\ee}^* \Phi_{\ee}^+ u), \qquad u \in C(\cG),
\]
where the right-hand side is identified with the constant function on \( \ee \).
Now we define the operator \( Q\colon C(\cG) \to C(\cG) \) by the formula
\begin{equation}
\label{eq:Q}
Q \coloneqq \sigma (Q_{\ee})_{\eE} = (\sigma_{\ee} Q_{\ee})_{\eE}.
\end{equation}
Finally, in order to characterize \( \X_0 = \range P \), note that
\( \range P_{p_{\ee},q_{\ee}} = \R \), provided that
\( p_{\ee} q_{\ee} \neq 1 \), and \( \range P_{p_{\ee},q_{\ee}} = \R^2 \),
provided that \( p_{\ee} q_{\ee} = 1 \).
Hence
\[
\range P \eqqcolon C_0(\cG),
\]
where \( C_0(\cG) \) is isometrically isomorphic to \( \R^{N+N_1} \), \( N_1 \)
being the number of \( \eE \) such that \( p_{\ee} = q_{\ee} = 1 \).
Note that the operator \( Q \) is bounded on \( C_0(\cG) \).
Now we are ready to state the main result.

\begin{theorem}
\label{thm:convergence-in-C}
Let \( A_\epsilon \), \( \epsilon \in (0,1] \) be defined by~\eqref{eq:Aeps} with domain~\eqref{eq:Adom}. For operators \( P \) and \( Q \) given by~\eqref{eq:P} and~\eqref{eq:Q}, respectively, it follows that
\[
\lim_{\epsilon \to 0^+} \e^{t A_\epsilon} u = \e^{tQ} Pu, \qquad t > 0,\ u \in
C(\cG)
\]
uniformly in \( t \) on compact subsets of \( (0,+\infty) \).
Moreover, if \( u \in C_0(\cG) \), then the formula holds also for \( t = 0 \),
and the convergence is uniform on compact subsets of \( [0,+\infty) \).
\end{theorem}

To verify conditions~\ref{item:kurtz2} and~\ref{item:kurtz4} and consequently
prove Theorem~\ref{thm:convergence-in-C}, we need the following result.

\begin{lemma}
\label{lem:res-conv}
For sufficiently large \( \lambda > 0 \) and all \( u \in C(\cG) \) we have
\[
\lim_{\epsilon \to 0^+} \res(\lambda,\epsilon A_\epsilon) u = \res(\lambda,A_0) u,
\]
and
\[
\lim_{\epsilon \to 0^+} \res(\lambda,A_\epsilon) u = \res(\lambda,Q) Pu.
\]
\end{lemma}

\begin{proof}
We first prove the second equality.
Combining Theorem~\ref{thm:feller-semigroup-interval} and
Corollary~\ref{cor:aeps-res}, we are left with investigating the (strong) limit
of \( L_{\lambda,\epsilon} \Phi_\epsilon \).
Let \( u = (u_{\ee})_{\eE} \in C(\cG) \).
Then, calculating as in the proof of Theorem~\ref{thm:main-feller}, we have
\[
L_{\lambda,\epsilon} \Phi_\epsilon u(x) = (v_\ee)_{\eE},
\]
where
\[
v_{\ee}(x) \coloneqq c_{\ee} \e^{\gamma_{\ee} x} + d_{\ee} \e^{-\gamma_{\ee} x},
\qquad x \in [0,1]
\]
for \( \gamma_{\ee} \coloneqq \sqrt{\lambda \epsilon/\sigma_{\ee}} \), and
\( c_{\ee} \), \( d_{\ee} \) being the unique pair of real numbers satisfying
the system of linear equations~\eqref{eq:system} for
\( z \coloneqq \gamma_{\ee} \), and
\( a \coloneqq \epsilon \gamma_{\ee}^{-1} \Phi_{\ee}^- u \),
\( b \coloneqq \epsilon \gamma_{\ee}^{-1} \Phi_{\ee}^+ u \).
Now, for each \( \eE \) we consider two cases, depending on whether
\( p_{\ee} = q_{\ee} = 1 \) or not.

\emph{Case} 1: Assume that \( p_{\ee} \neq 1 \) or \( q_{\ee} \neq 1 \).
Then, solving the system explicitly using Lemma~\ref{lem:det} and taking the
limit as \( \epsilon \to 0^+ \), we obtain that both \( c_{\ee} \) and
\( d_{\ee} \) converges to \( \delta_{\ee}/2 \), where
\[
\delta_{\ee} \coloneqq \lambda^{-1}\sigma_{\ee} \frac{q_{\ee}^* \Phi_{\ee}^- u + p_{\ee}^* \Phi_{\ee}^+ u}{1-p_{\ee}q_{\ee}}.
\]
Consequently \( v_{\ee} \) converges in \( C[0,1] \) to the constant function
that equals \( \delta_{\ee} \) on \( [0,1] \).

\emph{Case} 2: Assume that \( p_{\ee} = q_{\ee} = 1 \). 
Then, using Lemma~\ref{lem:det}, we have
\[
v_{\ee}(x) = \lambda^{-1}\sigma_{\ee} \Phi_{\ee}^- u \frac{\e^{\gamma_{\ee}
    (1-x)} - \e^{-\gamma_{\ee}(1-x)}}{\e^{\gamma_{\ee}} - \e^{-\gamma_{\ee}}} +
\lambda^{-1}\sigma_{\ee} \Phi_{\ee}^+ u \frac{\e^{\gamma_{\ee} x} -
  \e^{-\gamma_{\ee} x}}{\e^{\gamma_{\ee}} - \e^{-\gamma_{\ee}}}, \qquad x \in
[0,1],
\]
and taking \( \epsilon \to 0^+ \), we see that \( v_{\ee} \) converges in
\( C[0,1] \) to the function
\[
x \mapsto (1-x) \lambda^{-1}\sigma_{\ee} \Phi_{\ee}^- u + x
\lambda^{-1}\sigma_{\ee} \Phi_{\ee}^+ u.
\]

Combining Case 1 and Case 2 we see that \( L_{\lambda,\epsilon} \Phi_\epsilon \)
converges in \( C(\cG) \) to \( \lambda^{-1} Q \) as \( \epsilon \to 0^+ \).
Hence (recall that the norm of \( L_{\lambda,\epsilon} \) is uniformly bounded
in \( \epsilon \), see~\eqref{eq:Lnorm}), by~\eqref{eq:Azerolim} we obtain
\begin{align*}
\lim_{\epsilon \to 0^+} \res(\lambda,A_\epsilon) u &= \lambda^{-1} \lim_{\epsilon \to 0^+} (I_{C(\cG)} - L_{\lambda,\epsilon} \Phi_\epsilon)^{-1} P u\\
&= \lambda^{-1} (I_{C(\cG)} - \lambda^{-1} Q)^{-1} P u\\
&= (\lambda - Q)^{-1} P u
\end{align*}
for all \( u \in C(\cG) \), which completes the proof of the second part of the
lemma.

The first part follows similarly from the identity (obtained in the same way as
Corollary~\ref{cor:aeps-res})
\[
\res(\lambda,\epsilon A_\epsilon) = \res(1,L_{\lambda,1} \Phi_\epsilon) \res(\lambda,A_0),
\]
which holds for sufficiently large \( \lambda > 0 \).
\end{proof}

\section{Appendix: One-dimensional sticky diffusion}
\label{sec:one-dimens-sticky}

Fix \( p,q \in [0,1] \) and consider the linear operator \( G_{p,q} \) in
\( C[0,1] \) given by
\[
G_{p,q} f \coloneqq f'', \qquad f \in D(G_{p,q})
\]
with domain \( D(G_{p,q}) \) consisting of functions \( f \in C^2[0,1] \)
satisfying boundary conditions
\begin{equation}
\label{eq:boundary-segment}
pf''(0) - p^* f'(0) = qf''(1) + q^* f'(1) = 0,
\end{equation}
where
\[
p^* \coloneqq 1-p, \qquad q^* \coloneqq 1-q.
\]
We prove that \( G_{p,q} \) is the generator of a Feller semigroup in
\( C[0,1] \).
The stochastic process related to \( G_{p,q} \), in the sense of Theorem~3.15
in~\cite{liggett2010}, may be described as follows.
Inside \( (0,1) \) the process behaves like a Brownian motion with variance
\( 1 \).
When a Brownian particle hits \( 0 \), then the behaviour depends on \( p \):
\begin{itemize}
  \item If \( p = 0 \), then the barrier at \( 0 \) is reflecting, and the
particle bounces back to \( (0,1) \).
The time that this particle spends at \( 0 \) is of measure zero with respect to
the Lebesgue measure, however, it has a positive measure with respect to the
L\'{e}vy local time \( t^+ \), see~\cite[Section~4]{bobrowski-morawska} and the
references given there.
  \item If \( p = 1 \), then the barrier at \( 0 \) is absorbing, and the
particle stays at \( 0 \) forever.
  \item If \( p \in (0,1) \), then the barrier at \( 0 \) is sticky, and the
particle stays at \( 0 \) for some time depending on \( p \).
In contradistinction to the case \( p = 0 \), the time has positive Lebesgue
measure and increases with \( p \), see the first displayed formula after (3.42)
in~\cite[p.~128]{liggett2010}.
\end{itemize}
We call \( p \) a stickiness coefficient at \( 0 \).
Similar description is valid for the barrier at \( 1 \) and the stickiness
coefficient \( q \).

It is also possible to describe the behaviour of the process governed by
\( G_{p,q} \) as the time increases.
To this end we introduce \( P_{p,q} \) as the bounded linear operator in
\( C[0,1] \) defined for \( p \neq 1 \) and \( q \neq 1 \) by
\[
P_{p,q} f \coloneqq \frac{pq^*f(0) + p^*qf(1) + p^*q^* \int_0^1
  f}{pq^*+p^*q+p^*q^*}, \qquad f \in C[0,1];
\]
here the right-hand side is identified with the constant function.
For \( p = q = 1 \) we additionally set
\[
P_{1,1}f(x) \coloneqq (1-x)f(0) + xf(1), \qquad f \in C[0,1],\ x \in [0,1].
\]
Note that for all \( p,q \in [0,1] \) the operator \( P_{p,q} \) is a
projection, that is \( P_{p,q}^2 = P_{p,q} \).

The main result concerning one-dimensional diffusion with sticky barriers is as
follows.

\begin{theorem}
\label{thm:feller-semigroup-interval}
The operator \( G_{p,q} \) generates a conservative, bounded analytic Feller
semigroup in \( C[0,1] \) of angle \( \pi/2 \), and 
\begin{equation}
\label{eq:feller-limit}
\lim_{t \to +\infty} \e^{tG_{p,q}} f = P_{p,q} f, \qquad f \in C[0,1].
\end{equation}
\end{theorem}

Before we prove Theorem~\ref{thm:feller-semigroup-interval} we first state some
auxiliary results.
By \( \C_+ \) we denote the right-half of the complex plane, that is
\[
\C_+ \coloneqq \set{ \lambda \in \C\colon \re \lambda > 0 }.
\]

\begin{lemma}
\label{lem:det}
For every \( \mu \in \C_+ \), \( p,q \in [0,1] \) and functions
\( a,b\colon \C_+ \to \C \), the linear system
\begin{equation}
\label{eq:system}
\begin{bmatrix} p\mu - p^* & p\mu + p^*\\ \e^{\mu} \lrp{ q \mu +
  q^* } & \e^{-\mu} \lrp{ q \mu - q^* }.
\end{bmatrix}
\begin{bmatrix}
c_{\mu} \\ d_{\mu}
\end{bmatrix} =
\begin{bmatrix}
a(\mu) \\ b(\mu)
\end{bmatrix}
\end{equation}
has a unique complex solution \( (c_\mu,d_{\mu}) \) given by
\begin{equation}
\label{eq:clambda}
c_\mu = \frac{(p\mu+p^*)b(\mu) - \e^{-\mu}(q\mu-q^*)a(\mu)}{\e^{\mu}(p\mu+p^*)(q\mu+q^*)-\e^{-\mu}(p\mu-p^*)(q\mu-q^*)}
\end{equation}
and
\begin{equation}
\label{eq:dlambda}
d_{\mu} = \frac{\e^{\mu}(q\mu+q^*)a(\mu) - (p\mu-p^*)b(\mu)}{\e^{\mu}(p\mu+p^*)(q\mu+q^*)-\e^{-\mu}(p\mu-p^*)(q\mu-q^*)}
\end{equation}
Moreover,
\begin{equation}
\label{eq:det-estimate}
c_{\mu} = O(\abs{ a(\mu) }\e^{-2 \re\mu} + \abs{ b(\mu) }\e^{-\re \mu}),
\qquad d_{\mu} = O(\abs{ a(\mu) } + \abs{ b(\mu) }\e^{-\re \mu})
\end{equation}
as \( \re \mu \to +\infty \).
\end{lemma}

\begin{proof}
A~little bit of algebra shows that the determinant \( D_\mu \) of the system
equals
\begin{equation}
\label{eq:det-cd}
\e^{-\mu}(p\mu+p^*)(q\mu+q^*)
\lrb[\big]{H_p(\mu)H_q(\mu) - \e^{2\mu}},
\end{equation}
where
\[
H_r(z) \coloneqq \frac{rz - 1 + r}{rz + 1 - r}, \qquad r \in [0,1],\ z \in \C_+.
\]
We have
\[
\abs{H_r(z)} \le 1, \qquad r \in [0,1],\ z \in \C_+.
\]
Indeed, this is trivial for \( r = 0 \), and for fixed \( r \in (0,1] \) the
function \( H_r \) is (a restriction of) the M\"{o}bius transformation that maps
\( \C_+ \) onto the open unit ball in \( \C \).
Therefore, by~\eqref{eq:det-cd} it follows that
\[
\abs{ D_{\mu} } \ge \e^{-\re \mu} \abs{ p\mu+p^* } \abs{ q\mu +
  q^* } (\e^{2\re \mu} - 1) > 0,
\]
where in the last inequality we used the fact that \( \re \mu > 0 \).
Hence the system has a unique solution, and the remaining part of the lemma
follows by applying Cramer's rule.
\end{proof}

We are ready to prove that the resolvent of \( G_{p,q} \) exists.
To this end, for every \( \lambda \in \C \) we define the function \( \e_\lambda \) in
\( C[0,1] \) by
\[
\e_\lambda(x) \coloneqq \e^{\lambda x}, \qquad x \in [0,1].
\]
Moreover, for \( \lambda \in \C \setminus \set{ 0 } \) and \( f \in C[0,1] \) we
introduce \( h_\lambda = h_{\lambda,f} \in C[0,1] \) by the formula
\[
h_\lambda(x) \coloneqq \frac{1}{2\lambda} \int_0^1 \e^{-\lambda \abs{ x-y }}
f(y) \dd y, \qquad x \in [0,1].
\]

\begin{lemma}
\label{lem:res-exists}
The resolvent set of \( G_{p,q} \) contains \( \C \setminus (-\infty,0] \) and
\begin{equation}
\label{eq:explicit}
R(\lambda^2,G_{p,q}) f = c_\lambda \e_\lambda + d_\lambda \e_{-\lambda} +
h_{\lambda}, \qquad \lambda \in \C_+,\ f \in C[0,1], \qquad
\end{equation}
where \( c_{\lambda} \) and \( d_{\lambda} \) are given
by~\eqref{eq:clambda}-\eqref{eq:dlambda} with
\begin{equation}
\label{eq:albl}
a(\lambda) \coloneqq pf(0) \lambda^{-1} - h_{\lambda}(0) (p\lambda-p^*), \qquad
b(\lambda) \coloneqq qf(1) \lambda^{-1} - h_{\lambda}(1) (q\lambda - q^*).
\end{equation}
\end{lemma}

\begin{proof}
Since each complex number in \( \C \setminus (-\infty,0] \) has a~unique square
root with positive real part, to prove that \( \rho(G_{p,q}) \) contains
\( \C \setminus (-\infty,0] \) it is enough to show that the resolvent
\( R(\lambda^2,G_{p,q}) \) exists for all \( \lambda \in \C_+ \).
Hence, fix \( \lambda \in \C_+ \) and let \( f \in C[0,1] \).
Observe that the function \( h_{\lambda} \) belongs to \( C^2[0,1] \) and
\( \lambda^2 h_{\lambda} - h_{\lambda}'' = f \).
Therefore, each \( g \in C^2[0,1] \) satisfying \( \lambda^2 g - g'' = f \) may
be written in the form
\[
g = c_\lambda \e_\lambda + d_\lambda \e_{-\lambda} + h_\lambda
\]
for some \( c_\lambda,d_\lambda \in \C \).
It follows that \( g \) belongs to \( D(G_{p,q}) \),
see~\eqref{eq:boundary-segment}, if and only if~\eqref{eq:system} holds with
\( a \) and \( b \) as in~\eqref{eq:albl}.
Consequently, Lemma~\ref{lem:det} implies that \( g \in D(G_{p,q}) \) for the
unique pair \( (c_{\lambda}, d_{\lambda}) \) given
by~\eqref{eq:clambda}-\eqref{eq:dlambda} with \( a \) and \( b \) as
in~\eqref{eq:albl}.
Moreover, there is an \( M = M(\lambda) > 0 \) such that
\( \norm{ g }_{C[0,1]} \le M \norm{ f }_{C[0,1]} \), which proves that
\( \lambda^2 \) belongs to the resolvent set of \( G_{p,q} \) and
\( R(\lambda^2,G_{p,q}) = g \).
\end{proof}

Using the explicit formula obtained in Lemma~\ref{lem:res-exists}, we calculate
the strong limit of the resolvent \( R(\lambda,G_{p,q}) \) as \( \lambda \)
converges to \( 0 \).

\begin{proposition}
\label{prop:limit-at-0}
We have
\[
\lim_{\lambda \to 0} \lambda R(\lambda,G_{p,q}) f = P_{p,q} f, \qquad f \in
C[0,1],
\]
in \( C[0,1] \), where the limit is taken over
\( \lambda \in \C \setminus (-\infty,0] \).
\end{proposition}

\begin{proof}
By the same reason as in the beginning of the proof of
Lemma~\ref{lem:res-exists}, it suffices to find the limit of
\( \lambda^2 R(\lambda^2,G_{p,q}) f \) as \( \lambda \to 0 \) in \( \C_+ \).
Furthermore, since \( \lambda^2 h_\lambda \) converges to zero as
\( \lambda \to 0 \), by~\eqref{eq:explicit} we are left with calculating the
limit of \( \lambda^2 (c_\lambda \e_\lambda + d_\lambda \e_{-\lambda}) \).

We split the proof into two cases depending on whether both barriers are
absorbing, that is \( p=q=1 \), or not.

\emph{Case} 1: Suppose \( p \neq 1 \) or \( q \neq 1 \).
We have
\[
h_\lambda(0) = \frac{1}{2\lambda} \int_0^1 (1 + O(\lambda)) f(x) \dd x =
\frac{1}{2\lambda} \int_0^1 f + O(1),
\]
and similarly
\[
h_{\lambda}(1) = \frac 1{2 \lambda} \int_0^1 f + O(1);
\]
here, and in what follows in the proof, when we use the big-O notation we always
mean ``as \( \lambda \to 0 \) in \( \C_+ \).''
Hence, we may rewrite \( a(\lambda) \) and \( b(\lambda) \) given
by~\eqref{eq:albl} in the asymptotic form
\[
a(\lambda) = \frac{pf(0)}{\lambda} + \frac{p^*}{2\lambda} \int_0^1 f + O(1),
\qquad b(\lambda) = \frac{qf(1)}{\lambda} + \frac{q^*}{2\lambda} \int_0^1 f +
O(1).
\]
Denoting by \( D_\lambda \) the determinant of the coefficient matrix
in~\eqref{eq:system}, we have
\begin{align*}
\lambda c_\lambda &= \frac{- \e^{-\lambda} pq^*f(0) - \e^{-\lambda} \frac{p^*q^*}{2} \int_0^1 f -
  p^*qf(1) - \frac{p^*q^*}{2} \int_0^1 f + O(\lambda)}{D_\lambda}\\
&= - \frac{pq^*f(0) + p^*qf(1) + p^*q^* \int_0^1 f + O(\lambda)}{D_{\lambda}}.
\end{align*}
Since
\[
D_\lambda = (\e^{-\lambda} - \e^\lambda) p^*q^* - 2\lambda pq^* - 2\lambda p^*q
+ O(\lambda^2) = -2\lambda(p^*q^*+pq^*+p^*q) + O(\lambda^2),
\]
this leads to
\[
\lambda c_{\lambda} = \frac{1}{2\lambda} P_{p,q} f + O(1).
\]
Similarly, the same relation holds with \( c_\lambda \) replaced by
\( d_\lambda \), and consequently we easily check that
\[
\norm{ \lambda^2 (c_{\lambda} \e_{\lambda} + d_\lambda \e_{-\lambda}) - P_{p,q}
  f }_{C[0,1]} = O(\lambda).
\]

\emph{Case} 2: Suppose \( p = q = 1 \).
Then by Lemma~\ref{lem:res-exists} it follows that
\[
\lambda^2 c_{\lambda} = \frac{\e^{-\lambda} f(0) - f(1) - \lambda^2\e^{-\lambda} h_{\lambda}(0) +
  \lambda^2 h_{\lambda}(1)}{\e^{-\lambda} - \e^{\lambda}},
\]
and
\[
\lambda^2 d_{\lambda} = \frac{-\e^{\lambda} f(0) + f(1) + \lambda^2 \e^{\lambda} h_{\lambda}(0) -
  \lambda^2 h_{\lambda}(1)}{\e^{-\lambda} - \e^{\lambda}}.
\]
Denoting
\[
E_{\lambda}(x) \coloneqq \frac{\e^{-\lambda x} - \e^{\lambda x}}{\e^{-\lambda} - \e^{\lambda}},
\qquad x \in [0,1],
\]
for all \( x \in [0,1] \) we have
\begin{align*}
\lambda^2 \lrb{c_{\lambda} \e_{\lambda}(x) + d_{\lambda} \e_{-\lambda}(x)} = \lrb{f(0) - \lambda^2
  h_{\lambda}(0)} E_{\lambda}(1-x) + \lrb{f(1) - \lambda^2 h_{\lambda}(1)} E_{\lambda}(x).
\end{align*}
Since \( \sup_{x \in [0,1]} \abs{E_{\lambda}(x) - x} = O(\lambda) \), this leads
to
\[
\norm{ \lambda^2 (c_{\lambda} \e_{\lambda}(x) + d_{\lambda} \e_{-\lambda}) -
  P_{1,1} f }_{C[0,1]} = O(\lambda),
\]
which completes the proof.
\end{proof}

\begin{proposition}
\label{prop:sectorial-interval}
The operator \( G_{p,q} \) is sectorial with angle \( \pi/2 \).
\end{proposition}

\begin{proof}
By Lemma~\ref{lem:res-exists} we are left with proving that for each
\( \delta \in (0,\pi/2) \) there exists \( M > 0 \) such that
\begin{equation}
\label{eq:res-sectoriality}
\norm{ R(\lambda,G_{p,q}) }_{\lino{C[0,1]}} \le \frac M{\abs{ \lambda }}, \qquad
\lambda \in \Sigma_{\pi/2 + \delta}.
\end{equation}
Since the resolvent is an analytic function on the resolvent set,
see~\cite[Proposition~IV.1.3]{MR1721989}, by Proposition~\ref{prop:limit-at-0}
it follows that the function
\( \lambda \mapsto \lambda \norm{ R(\lambda,G_{p,q}) }_{\lino{C[0,1]}} \) is
bounded in every bounded subset of \( \Sigma_{\pi/2+\delta} \).

We estimate the resolvent ``at infinity''.
Note that it suffices to prove~\eqref{eq:res-sectoriality} in the sector
\( \Sigma_{\pi/4+\delta/2} \) with \( \lambda \) replaced by \( \lambda^2 \).
Let \( f \in C[0,1] \) be such that \( \norm{ f }_{C[0,1]} \le 1 \).
Lemma~\ref{lem:res-exists} implies that
\begin{equation}
\label{eq:res-asy}
\norm{ R(\lambda^2,G_{p,q}) f }_{C[0,1]} \le \abs{ c_\lambda } \e^{\re \lambda}
+ \abs{ d_\lambda } + \norm{ h_\lambda }_{C[0,1]}, \qquad \lambda \in \C_+.
\end{equation}
For every \( \lambda \in \Sigma_{\pi/4+\delta/2} \) we have
\( \re \lambda \ge \cos(\pi/4+\delta/2) \abs{ \lambda } \), hence
\[
\norm{ h_\lambda }_{C[0,1]} \le \frac{1}{2 \abs{ \lambda } } \sup_{x \in [0,1]}
\int_0^1 \e^{-\re \lambda \abs{ x-y }} \dd y = O(\lambda^{-2})
\]
uniformly in \( f \); here, and in what follows in the proof, when we use the
big-O~notation we always mean ``as \( \abs{ \lambda } \to +\infty \) in the
sector \( \Sigma_{\pi/4+\delta/2} \).''
This leads to
\[
c_\lambda = \frac{O(\lambda^{-2})}{\e^\lambda + O(1)} = O(\lambda^{-2}
\e^{-\lambda}), \qquad d_\lambda = \frac{O(\lambda^{-2}
  \e^{\lambda})}{\e^\lambda + O(1)} = O(\lambda^{-2})
\]
uniformly in \( f \).
Using~\eqref{eq:res-asy} we obtain
\( \norm{R(\lambda^2,G_{p,q})}_{\lino{C[0,1]}} = O(\lambda^{-2}) \) as desired.
\end{proof}

\begin{lemma}
\label{lem:dense-derivative}
Let \( a_0 \), \( a_1 \), \( b_0 \) and \( b_1 \) be real numbers.
For every \( \epsilon > 0 \) there exists a function \( f \in C^2[0,1] \)
satisfying
\begin{align}
f'(0) &= a_0, & f''(0) &= b_0,\label{eq:d-0}\\
f'(1) &= a_1, & f''(1) &= b_1,\label{eq:d-1}
\end{align}
and such that
\[
\norm{ f }_{C[0,1]} \le \epsilon.
\]
\end{lemma}

\begin{proof}
Let \( f_{\gamma} \in C[0,1] \) be given by
\begin{equation}
\label{eq:fgamma}
f_{\gamma}(x) \coloneqq \alpha_0 \e^{-\gamma x} + \beta_0 \e^{-\gamma^2 x} +
\alpha_1 \e^{-\gamma(1-x)} + \beta_1 \e^{-\gamma^2(1-x)}, \qquad x \in [0,1],
\end{equation}
where \( \alpha_0 \), \( \alpha_1 \), \( \beta_0 \) and \( \beta_1 \) are real
numbers.
Such \( f_{\gamma} \) satisfies~\eqref{eq:d-0}--\eqref{eq:d-1} if and only if
\begin{equation}
\label{eq:systemgamma}
\begin{bmatrix}
-\gamma & -\gamma^2 & \gamma \e^{-\gamma} & \gamma^2 \e^{-\gamma^2}\\
\gamma^2 & \gamma^4 & \gamma^2 \e^{-\gamma} & \gamma^4 \e^{-\gamma^2}\\
-\gamma \e^{-\gamma} & -\gamma^2 \e^{-\gamma^2} & \gamma & \gamma^2\\
\gamma^2 \e^{-\gamma} & \gamma^4 \e^{-\gamma^2} & \gamma^2 & \gamma^4
\end{bmatrix}
\begin{bmatrix}
\alpha_0\\
\beta_0\\
\alpha_1\\
\beta_1
\end{bmatrix}
=
\begin{bmatrix}
a_0\\
b_0\\
a_1\\
b_1
\end{bmatrix}.
\end{equation}
For the determinant \( D_\gamma \) of the coefficients matrix we have
\[
D_{\gamma} = -\gamma^{10} + O(\gamma^8) \qquad \text{as
  \( \gamma \to +\infty \)}.
\]
Hence, for sufficiently large \( \gamma_0 > 0 \), it follows that
\( D_{\gamma} \neq 0 \) for \( \gamma > \gamma_0 \).
Consequently, for all \( \gamma > \gamma_0 \) we choose \( \alpha_0 \),
\( \alpha_1 \), \( \beta_0 \) and \( \beta_1 \) in such a way that
\( f_{\gamma} \) given by~\eqref{eq:fgamma}
satisfies~\eqref{eq:d-0}--\eqref{eq:d-1}.
Moreover, by~\eqref{eq:systemgamma} it follows that
\[
\alpha_0 = \alpha_1 = O(\gamma^{-1}), \qquad \beta_0 = \beta_1 = O(\gamma^{-3})
\qquad \text{as \( \gamma \to +\infty \).}
\]
Therefore, \( \lim_{\gamma \to +\infty} \norm{ f_{\gamma} }_{C[0,1]} = 0 \),
which completes the proof.
\end{proof}




\begin{proof}[Proof of Theorem~\textnormal{\ref{thm:feller-semigroup-interval}}]
By Lemma~\ref{lem:dense-derivative} it follows that the operator \( G_{p,q} \)
is densely defined.
Hence, by Proposition~\ref{prop:sectorial-interval}, it generates an analytic
semigroup \( \set{\e^{tG_{p,q}}}_{t \ge 0} \) in \( C[0,1] \).
We prove that \( G_{p,q} \) is a Feller generator.
By a~well-known characterization of Feller semigroups, see~\cite[Theorem~2.2,
p.~165]{MR838085}, it suffices to check that \( G_{p,q} \) satisfies the positive
maximum principle, that is: if \( f \in D(G_{p,q}) \) attains the maximum at
\( x \in [0,1] \), then \( f(x) \ge 0 \) implies \( G_{p,q}f(x) \le 0 \).
This is clear for \( x \in (0,1) \), and if \( f \) attains the nonnegative
maximum at \( x = 0 \) or \( x = 1 \), then \( f'(0) \le 0 \) or
\( f'(1) \ge 0 \), respectively.
Thus the claim follows, since \( f \) satisfies boundary
conditions~\eqref{eq:boundary-segment}.
Moreover, the function \( 1_{[0,1]} \) belongs to the domain of \( G_{p,q} \)
and \( G_{p,q} 1_{[0,1]} = 0 \), hence the semigroup is conservative.

To prove~\eqref{eq:feller-limit} note first that the limit on the left-hand side
exists, which follows by the sectoriality of the semigroup and
Proposition~\ref{prop:limit-at-0},
see~\cite[Corollary~32.1]{bobrowski-convergence}.
Let \( f \in C[0,1] \), and write
\[
f = (f - P_{p,q} f) + P_{p,q} f.
\]
By Proposition~\ref{prop:limit-at-0} we have
\( \lambda R(\lambda,G_{p,q}) (f - P_{p,q} f) \to P_{p,q}(f - P_{p,q}f) = 0 \)
as \( \lambda \) converges to zero in \( \C \setminus (-\infty,0] \).
However, the resolvent \( R(\lambda,G_{p,q}) \) is the Laplace transform of the
semigroup generated by \( G_{p,q} \), hence
\[
\lambda R(\lambda,G_{p,q}) = \int_0^{+\infty} \e^{-t} \e^{\frac{t}{\lambda}
  G_{p,q}} \dd t,
\]
which implies that \( \e^{tG_{p,q}}(f - P_{p,q} f) \to 0 \) as
\( t \to +\infty \).
Finally, because \( G_{p,q} P_{p,q} f = 0 \) for all \( p,q \in [0,1] \), we
have \( R(\lambda,G_{p,q}) P_{p,q} f = \lambda^{-1} P_{p,q} f \).
Therefore, by the Yosida approximation it follows that
\( \e^{tG_{p,q}} P_{p,q} f = P_{p,q}f \), which completes the proof.
\end{proof}

\bibliographystyle{amsplain}
\bibliography{references}

\end{document}